\documentclass[12pt]{amsart}
\usepackage[margin=1in]{geometry}

\usepackage{amsmath}
\usepackage{amssymb}
\usepackage{commath}
\usepackage{mathtools}

\usepackage{hyperref}
\theoremstyle{plain}
\newtheorem{theorem}{Theorem}
\newtheorem{lemma}[theorem]{Lemma}
\newtheorem{proposition}[theorem]{Proposition}
\newtheorem{corollary}[theorem]{Corollary}

\newtheorem{remark}[theorem]{Remark}
\numberwithin{theorem}{section}
\numberwithin{equation}{section}

\newcommand{\C}{\mathbb{C}}
 \newcommand{\R}{\mathbb{R}}  

\newcommand{\cO}{\mathcal{O}}

\newcommand{\wG}{\widetilde{G}}

\begin{document}

\title[Atomic Decomposition]{New atomic decompositions of Bergman spaces on bounded symmetric domains}

\author{Jens Gerlach Christensen} 
\address{ Department of Mathematics, Colgate University} 
\email{jchristensen@colgate.edu}
\urladdr{http://www.math.colgate.edu/~jchristensen}

\author{Gestur \'Olafsson} 
\address{ Department of Mathematics, Louisiana State University} 
\email{olafsson@math.lsu.edu}
\urladdr{http://www.math.lsu.edu/~olafsson}
\thanks{The research was partially supported by NSF grant   DMS 1321794 during the MRC program 
  \textit{Lie Group Representations, Discretization, and Gelfand Pairs.}
  The research of G. \'Olafsson was also partially supported
by Simons grant 586106} 

\begin{abstract}
  We provide a large family of atoms for Bergman spaces on irreducible bounded
  symmetric domains.   The atomic decompositions are derived using the holomorphic
  discrete series representations for the domain, and the approach
  is inspired by recent advances in wavelet and coorbit theory.
  Our results vastly generalize previous work  by Coifman and Rochberg.
  Their  atoms correspond to
    translates of a constant function at a discrete subset of the
    automorphism group of the domain. In this paper we show that
    atoms can be obtained as translates
    of any holomorphic function with rapidly decreasing coefficients
    (including polynomials).
  This approach also settles the relation between atomic decompositions
  for the bounded and unbounded realizations of the domain.
\end{abstract}

\maketitle

\section{Introduction}
This paper is concerned with providing
atomic decompositions of Bergman spaces on bounded symmetric domains.
The results extend similar results by Coifman and Rochberg
\cite{Coifman1980} carried out for Bergman spaces on the unbounded realization of the domain
and on the unit ball. Coifman and Rochberg asked if their decompositions
would hold for the bounded domains, and in this paper we give a positive answer to this
question. Moreover, we rectify an issue occuring in higher rank spaces which
was pointed out in a remark on p. 614 in \cite{Bekolle1986}
(see also Remark 4 in \cite{Faraut1990}).
While an extension to the bounded domains
was predicted by Faraut and Koranyi in the introduction of
\cite{Faraut1990}
our results provide a much larger class of
atoms than have previously been discovered.
The usual atomic decompositions of Bergman spaces
arise from a discretization
of the integral reproducing formula, and atoms can thus be regarded
as samples of the Bergman kernel in one of the variables.
It turns out that this result can be formulated in terms of 
the holomorphic discrete series representation, in which case
the classical atoms correspond to letting a discrete subset
of the group of isometries act on a constant function.
This viewpoint is extended widely in this paper where we
show that atoms can be obtained by translates of any polynomial, 
or more generally, any analytic function with rapidly decreasing
coefficients.  Similar results have been obtained
by Pap \cite{Pap2012} on the unit disc (using the Blaschke group)
and by the authors and their collaborators
\cite{Christensen2017,Christensen2019} in the case of the unit ball. 

In the last section of the article we  use representation theory to   explain
the relation between atomic decompositions of Bergman spaces in
the bounded and the unbounded realization. 
This issue was unresolved in \cite{Coifman1980} and therefore
their Theorem 2 did not transfer directly to bounded domains for
all parameters (see beginning of \S 2 on p. 14 in \cite{Coifman1980}).
Representation theory lifts the construction of atomic decompositions
from discretizing a reproducing
formula on the domain to discretizing a convolution
reproducing formula of matrix coefficients for the group of isometries.
In particular, we use that the matrix coefficients have the same decay for
the two realizations of the domain.

The motivation for our  approach  comes from coorbit theory for 
invariant Banach spaces of functions and distributions which was
initiated by Feichtinger and Gr\"ochenig \cite{Feichtinger1988,Feichtinger1989}
and subsequently generalized in
\cite{Christensen1996,Christensen2009,Christensen2010,Christensen2011,Christensen2017,Christensen2019}
in order to treat a wider class of (projective) representations. In this paper we have
chosen to state most results without referring to 
coorbit theory. The main reason for doing so is 
that some results can be stated with fewer restrictions than if coorbit theory were used, and also the paper
would be a lot longer if we had to introduce the entire coorbit machinery.

\section{Main results}
In this section we state the main results of this article. The following sections will then be devoted to
proving those statements. We also introduce the minimal set of notation that will be
needed to formulate the statements. 

Let $D$ be a bounded symmetric domain in $\mathbb{C}^n$
containing the origin $o$. Let $G$ be a connected Lie group locally isomorphic with the group of
isometries of $D$. Then $G$ acts transitively on $D$ and $D=G/K$ where $K$ is
the stabilizer of $o$, $K=\{g\in G\mid g\cdot o=o\}$.  If $G$ is linear as we will always assume,
or more generally, with finite center, then $K$ is compact.

If nothing else is stated, let $dz$ denote the normalized Euclidean measure on $D$ and define
$A^2(D)$ to be the Hilbert space of  holomorphic functions $f:D\to \mathbb{C}$ for which
$$
\| f\|_{A^2} = \left( \int_D |f(z)|^2\,dz \right)^{1/2}<\infty.
$$
Thus $A^2(D)=\cO (D)\cap L^2(D,dz)$. The space
$A^2(D)$ is a reproducing kernel Hilbert space. Denote
the reproducing kernel (Bergman kernel) by $K(z,w)$ so that
$$
f(z) = \int_D f(w) K(z,w)\,dw
$$
for $f\in A^2(D)$. The map $K(z,w)$ is holomorphic in the first variable and
antiholomorphic in the second variable. Furthermore $K(w,z)=\overline{K(z,w)}$.

Let $J(x,z)$ be the complex Jacobian of the $G$-action on $D$ at the point
$z$. Then $z\mapsto J(x,z)$ is holomorphic for every $x\in G$,
\begin{equation}\label{eq:TransForm}
\int_D f(x\cdot z)|J(x,z)|^2dz = \int_D f(z)dz\quad\text{for all } f\in L^1(D),
\end{equation}
and the chain rule implies that
\begin{equation}\label{eq:ChainRule}
J(xy,z)=J(x,y\cdot z)J(y,z).
\end{equation}
In particular, since $J(e,z)=1$ as $e$ acts trivally, it follows that
\begin{equation}\label{eq:resipro}
J(x^{-1},x\cdot z)=J(x,z)^{-1} .
\end{equation}
Furthermore it follows that if $f\in A ^2 (D)$ then $J(x,z)f(x\cdot z)$ is also in $A^2(D)$ which implies by simple
calculation using the reproducing property that,
see Lemma \ref{lem:2.1},
\begin{equation}\label{eq:TransRule}
J(x,z)\overline{J(x,w)}K(x\cdot z,x\cdot w) = K(z,w) .
\end{equation}

The genus of $D$ is the number
$g:=(n+n_1)/r$
where $n=\dim_\C D$, $n_1$ is the complex dimension of the maximal complex subdomain of $D$ of
tube type and $r$ is the rank of $D$, the dimension of a maximal totally geodesic euclidean submanifold.
Lastly, let $a$ be a structural constant which satisfies $r(r-1)\frac{a}{2} = n_1 - r$.
Define the kernel $h(z,w) = K(z,w)^{-1/g}$, and
by abuse of notation let $h(z) = h(z,z)$.
For $x,y\in G$ also let
$h(x,y)=h(x\cdot 0,y\cdot 0)$ and $h(x) = h(x,x)$.
Note that by definition $h(x,y)$ is right $K$-invariant.

The function $ h(z,z)^{\gamma - g} =K(z,z)^{1-\gamma / g}$ is integrable
if and only if $\gamma >g-1$. For $\gamma >g-1$  let $c_\gamma^{-1}=\int_D h (z,z) ^{\gamma -g}dz$ and
$d\mu_\gamma (z)=c_\gamma h  (z,z)^{\gamma -g}\, dz$. Then $\mu_\gamma$ is a quasi-invariant 
probability measure on $D$.

Let $L^p_\gamma(D)$ be the space of equivalence classes of
measurable functions on $D$ for which
$$
\| f\|_{L^p_\gamma(D) } := \left( \int_D |f(z)|^p d\mu_\gamma (z)\right)^{1/p}< \infty.
$$
Define the weighted Bergman space
$A^p_\gamma (D)$ to be the subspace 
of holomorphic functions in $L^p_\gamma(D)$, i.e. $A_\gamma^p(D)=L^p_\gamma (D)\cap \cO (D)$ with norm inherited from $L^p_\gamma(D)$. We note that
the space of holomorphic polynomials $P[\C^n]$ is dense in $A_\gamma^p(D)$. The spacet $A_\gamma^p(D)$ is
a reproducing kernel Banach space.
Denote the reproducing kernel by $K_\gamma(z,w)$, which then satisfies
$$
f(z) = \int_D f(w)K_\gamma(z,w)\, d\mu_\gamma (w) 
$$
for $f\in A^p_\gamma(D)$. See for example Theorem 3 in \cite{Stoll1977}. We
have
\[K_\gamma (z ,w)=h(z , w)^{-\gamma} = K(z,w)^{\gamma/g}.\]

Define $j_\gamma (x,z)=J(x,z)^{\gamma /g}$. Note that $j_\gamma (x,z)$ is not necessarily defined globally on
the group $G$, but it is always defined on the universal covering group $\widetilde{G}$ of $G$. The function $j_\gamma$ satisfies
the cocycle relation $j_\gamma (xy,z)=j_\gamma (x,y\cdot z)j_\gamma (y,z)$. Furthermore
\[j_\gamma(x,z)\overline{j_\gamma (x,w)}K_\gamma (x\cdot z,x\cdot w)=K_\gamma (z,w) .\]
 
 This implies that  
$$
\pi_\gamma(x)f(z) = j_\gamma (x^{-1},z ) f(x^{-1}\cdot z).
$$
defines a representation of $\wG$ on $A_\gamma^2$ (or a projective representation
of $G$) and (\ref{eq:TransForm}) implies that the representation is unitary. 
It is well known that it is also irreducible. 

For a holomorphic function $f$ on $D$ decompose it into
homogeneous polynomials $f_k$ of degree $k$, i.e. 
$$
f = \sum_{k\geq 0} f_k.
$$
Define the holomorphic functions with rapidly decreasing coefficients $S_\gamma$
by
$$
S_\gamma = \{ f = \sum_{k\geq 0} f_k \mid \forall N,\, \exists C: 
\| f_k\|_{A^2_\gamma} \leq C(1+k)^{-N}  \},
$$
and the space of holomorphic functions with moderately growing coefficients
$S_\gamma^*$ by
$$
S^*_\gamma = \{ f=\sum_{k\geq 0} f_k \mid \exists N,C: 
\| f_k\|_{A^2_\gamma} \leq C(1+k)^{N}  \}.
$$
The space $S_\gamma$ is invariant under $\pi_\gamma$, and
$\pi_\gamma$ is a projective representation
of $G$ on $S_\gamma$. The space $S^*_\gamma$ is the dual of $S_\gamma$ and the dual pairing
of $f\in S^*_\gamma$ and $g\in S_\gamma$
is given by
$$
\langle f,g\rangle_\gamma = \sum_{k\geq 0} \int_D f_k(z)\overline{g_k(z)}\,d\mu_\gamma(z).
$$
For $\psi\in S_\gamma$ define the wavelet transform
$W_\psi^\gamma: S_\gamma^* \to C(\widetilde{G})$
by
$$
W_\psi^\gamma(f)(x) = \langle f,\pi_\gamma(x)\psi \rangle_\gamma.
$$
Notice that the function $|W^\gamma_\psi(f)|$ defines a continuous
function on $G$, so we will often allow ourselves to
write $|W_\psi^\gamma(f)(x)|$ for $x\in G$.
Also, the wavelet transform
is injective if $\psi$ is non-zero, due to the fact that
the representation $\pi_\gamma$ restricted to $S_\gamma$ is irreducible.
This follows since $S_\gamma$ are the smooth
vectors for the representation $\pi_\gamma$ by \cite{Chebli2004}, and 
by \cite[Proposition 2.6]{Bruhat1956} $\pi_\gamma$
restricted to the smooth vectors is irreducible if and only
if $\pi_\gamma$ is irreducible on $A_\gamma^2(D)$.

Define $L^p_\alpha(G)$ as the space of equivalence classes
of functions on the group $G$ for which
$$
\| f\|_{L^p_\alpha(G) } :=\left( \int_G |f(x)|^p h(x)^{\alpha}\,d\mu_G(x) \right)^{1/p}<\infty
$$
where $d\mu_G(x)=dx$ denotes the suitably normalized left-invariant Haar measure on $G$.

We are now ready to present the first result of this paper, which
gives a wavelet characterization of the Bergman spaces.
\begin{theorem}
  \label{thm:waveletcharbergman}
  A function $f$ is in the Bergman space $A^p_\alpha(D)$
  if and only if $f\in S_\gamma^*$ and 
  $W_\psi^\gamma (f) \in L^p_{\alpha -\gamma p/2}(G)$
  and either of the following two conditions are satisfied
  \begin{enumerate}
  \item $\gamma,\alpha >g-1$ and
    and $\psi$ is a non-zero constant,
  \item $\gamma > g-1 + (r-1)\frac{a}{2}$ and
    $g-1 - (r-1)\frac{a}{2} + p(r-1)\frac{a}{2} 
    < \alpha 
    < g-1 -(r-1)\frac{a}{2} + p(\gamma-g+1)$ and $\psi \in S_\gamma$ is non-zero.
  \end{enumerate}
  Moreover, the norms $\| f\|_{A^p_\alpha}$ and $\| W^\gamma_\psi(f)\|_{L^p_{\alpha-\gamma p/2}}(G)$ are equivalent.
\end{theorem}

\begin{remark}
  \label{rem:waveletchar}
  (1)
  If the rank of the space is large, then $\gamma$ has to be chosen
  sufficiently large in order to have the second part of the wavelet
  description.
  
  (2) We conjecture that it is possible to avoid this restriction and 
  to show that the Bergman spaces $A^p_\alpha$ have
  a wavelet characterization for the entire range $g-1 < \alpha < p(\gamma-g+1)+g-1$
  for general $\psi\in S_\gamma$.

  (3)
  Due to the coorbit theory for projective representations
  in \cite{Christensen2019} the collection
  \begin{equation}\label{eq:coorbits}
    \mathrm{Co}_\psi^\gamma L^p_{\alpha -\gamma p/2}(G) =\{
    f\in S_\gamma^* \mid W_\psi^\gamma(f) \in L^p_{\alpha -\gamma p/2}(G) \}
  \end{equation}
  is a non-trivial Banach space for the entire range
  of $g-1<\alpha< g-1+p(\gamma - g +1)$ when $\psi$ is a polynomial.
  By Theorem~\ref{thm:waveletcharbergman} these spaces are Bergman spaces
  for $\alpha$ in the smaller interval from
  Theorem~\ref{thm:waveletcharbergman}, and we expect that this is also
  true in the larger interval.

  (4)
  Similarly the coorbits (\ref{eq:coorbits}) are non-empty
  Banach spaces for $\psi\in S_\gamma$ when
  $g-1-(r-1)a/2<\alpha < g-1-(r-1)a/2+p(\gamma - g +1)$.
  We expect this to be true for the entire range
  $g-1 <\alpha < g-1+p(\gamma - g +1)$.
  In order to improve on this result
  we need a better understanding of the behaviour of wavelet coefficients
  $W^\gamma_\psi(\phi)$ for general $\psi,\phi\in S_\gamma$ than
  we use in this paper (Proposition~\ref{matrixdecay}).
\end{remark}

For a fixed countable collection of points $\{ x_i\}_{i\in I}$ in $G$
define the weighted
sequence space $\ell^p_\alpha(I)$ as the sequences $\{ \lambda_i\}$
for which the norm
$$
\| \{ \lambda_i\}\|_{\ell^p_\alpha} := \left( \sum_{i\in I} |\lambda_i|^p h(x_i)^\alpha\right)^{1/p}
$$
is finite.
We are now ready to state the main result of this paper which provides
atomic decompositions for Bergman spaces with atoms $\pi_\gamma(x_i)\psi$ for appropriately chosen $x_i$ in $G$.
\begin{theorem}
  \label{thm:atomicdecompbergman}
  Assume that $\gamma>g-1+(r-1)\frac{a}{2}$ and $\psi\in S_\gamma$
  is non-zero. If  
  \[g-1 +(p-1)(r-1)\frac{a}{2} 
    < \alpha 
    < g-1 + p(\gamma-g+1)-(r-1)\frac{a}{2},\]
  there is a countable discrete collection of points
  $\{ x_i\}_{i \in I}$ in $G$
  and associated functionals $\{ \lambda_i\}$ on
  $A^p_\alpha$ 
  such that every $f\in A^p_\alpha$ can be written
  $$
  f = \sum_{i\in I} \lambda_i(f) \pi_\gamma(x_i)\psi
  $$
  with $\| \{ \lambda_i(f)\}\|_{\ell^p_{\alpha-\gamma p/2}}
  \leq \| f\|_{L^p_\alpha(D)}$.
  Moreover, if $\{ c_i\} \in \ell^p_{\alpha-\gamma p/2}$ then
  $$
  g := \sum_{i\in I} c_i \pi_\gamma(x_i)\psi
  $$
  is in $A^p_\alpha$ and 
  $\| g\|_{L^p_\alpha(D)} \leq \| \{ \lambda_i(g)\}\|_{\ell^p_{\alpha-\gamma p/2}}$.
\end{theorem}

  \begin{remark}
  Notice that the range of parameters in Theorem~\ref{thm:atomicdecompbergman}
  is smaller than the range of parameters for which atomic decompositions
  have been found for tube type domains. See \cite{Christensen2018} in which
  atomic decompositions are constructed for tube type domains $T=V+i\Omega$
  by taking Laplace extensions of Besov spaces
  of distributions supported on the cone $\Omega$. 
  These atoms do not arise as samples of the Bergman kernel.
  Also see the Arxiv preprint \cite{Bekolle2017} where
  atoms are obtained for the same spaces but with atoms determined by
  the Bergman kernel.
\end{remark}

\begin{remark}
  In the case of $\psi=1$ it is possible to choose the points $x_i$ in a
  solvable subgroup of $G$ on which $\pi_\gamma$ is a representation.
  In this case the atomic decomposition becomes
  $$
  f(z)
  = \sum_{i\in I} \lambda_i(f) \left(\pi_\gamma(x_i)\psi\right)(z)
  = \sum_{i\in I} \lambda_i(f) \overline{j_\gamma (x_i,o)}
  K_\gamma (z,x_i\cdot o).$$
  By choosing $\alpha=g$ and $\gamma = 2g/p$
  we recover Theorem 1 of Coifman and Rochberg \cite{Coifman1980}
  for
  $$
  1\leq p < 1+ \max\left\{ \frac{2}{(r-1)a},\frac{g-(r-1)a/2}{g-1} \right\}.
  $$
  Namely, there are functionals $\widetilde{\lambda}_i$ such
  that $\| \{\widetilde{\lambda}_i(f)\} \|_{\ell^p} \leq \| f\|_{A^p(D)}$
  and
  $$
  f(z)= \sum_{i\in I} \widetilde{\lambda}_i(f)
  \left( \frac{K^2(z,w_i)}{K(w_i,w_i)}\right)^{1/p}
  $$
  for $w_i=x_i\cdot o$.
  Here the functionals
  $\widetilde{\lambda}_i$ differ from $\lambda_i$ by unimodular factors.

  For a parameter $\theta$ we can let $\gamma = (2\alpha + \theta g)/p$.
  Then we get the expansion from Theorem 2 in \cite{Coifman1980}
  for general bounded domains. Namely,
  there are functionals $\widehat{\lambda}_i$ such
  that $\| \{\widehat{\lambda}_i(f)\} \|_{\ell^p}
  \leq \| f\|_{A_\alpha^p(D)}$
  and
  $$
  f(z)= \sum_{i\in I} \widehat{\lambda}_i(f)
  \left( \frac{K^2(z,w_i)}{K(w_i,w_i)}\right)^{\frac{\alpha}{gp}}
  \left( \frac{K(z,w_i)}{K(w_i,w_i)}\right)^{\theta/p}
  $$
  for $w_i=x_i\cdot o$.
  Here the functionals $\widehat{\lambda}_i$
  differ from $h(w_i)^{-\frac{\theta g}{2p}}\lambda_i$ by
  unimodular factors.
\end{remark}

The special case of the unit ball in $\C^n$ was treated in
\cite{Christensen2019} and this article follows the same strategy and we will usually
use results from that article as far as the statements are stated in full generality or
the generalization is trivial.
\section{Background on Bergman spaces}

For $z\in D$ let $x_z\in G$ be such that $x_z\cdot 0=z$.
The element $x_z$ is not
unique because $x\cdot z=(xk)\cdot 0$ for all $k\in K$. 
The following is well known but we give a short discussion of the proof
as this will be used quite often in the following:

\begin{lemma}\label{lem:2.1}
  For $z,w\in D$ and $x\in G$ we have
  
  \noindent
  a) $J(x^{-1},x\cdot z)=J(x,z)^{-1}$.
  \smallskip
  
  \noindent
  b) $\displaystyle J(x,z)\overline{J(x,w)}K(x\cdot z,x\cdot w) =K(z,w)$.
  \smallskip
  
  \noindent
  c) $h(x) =|J(x,0)|^{2/g}$. 
  \smallskip
  
  \noindent
  d) The $G$-invariant measure on $D$ is, up to a positive constant, given by
\[
  f\mapsto \int_D f(z)h(z)^{-g} dz = \int_D f(z)K(z,z) dz , \quad f\in C_c(D).
\]
   \smallskip
  
  \noindent
e) The function $h$ is $K$-biinvariant and $h(x^{-1}) = h(x)$ for all $x\in G$.
\smallskip

\noindent
f) If $U$ is a compact subset of $G$ containing the identity then there exist constant $0<C_1<C_2$ such 
that
\[C_1h(x) \le \sup_{u\in U} h(xu)\le C_2 h(x)\quad \text{for all } x\in G.\]
\end{lemma}

\begin{proof}
a) We have $J(e,z)=1$. Hence the cocylce relation gives
\[1=J(x^{-1}x,z)=J(x^{-1},x\cdot z)J(x,z).\]

\noindent
b) Let $f\in A^2(D)$. Then $z\mapsto J(x,z)f(x\cdot z)=J(x^{-1},x\cdot z)f(x\cdot z)$ is again in
$A^2(D)$ as $z\mapsto J(x,z)$ is holomorphic and
\[\int_D |J(x,z)f(x\cdot z)|^2dz= \int_D |f(x\cdot z)|^2|J(x,z)|^2dz= \|f\|_{A^2}^2<\infty \]
The reproducing property of $K(z,w)$ gives:
\begin{align*}
 J(x,z)f(x\cdot z)&=\int_D f(x\cdot w)J(x,w)K(z,w)dw \\
 &=\int_D f(x\cdot w) \overline{J(x^{-1},x\cdot w)} K(z,x^{-1}x\cdot w) |J(x,w)|^2dw\\
 &=\int_ D f(w)\overline{J(x^{-1},w)}K(z,x^{-1}\cdot w)dw 
 \end{align*}
 But this can also be calculated as
 \[J(x,z)f(x\cdot z)=\int_D f(w)J(x,z)K(x\cdot z,w)dw .\]
 Replacing $w$ by $x\cdot w$ and using that $J(x^{-1},x\cdot w)= J(x,w)^{-1}$ implies the claim. 
 \smallskip
 
 \noindent
 c) By (b) and by the definition of $h$ as $h(z,z)=K(z,z)^{-1/g}$ we get 
 \[h(x\cdot 0,x\cdot 0)=K(x\cdot 0,x\cdot 0)^{-1/g}=|J(x,0)|^{2/g}.\]
  \smallskip
 
 \noindent
d) This follows easily from (b).
  \smallskip
 
\noindent
e) This is well know, but let us give a proof. That $h$ is $K$-biinvariant follows from (c) and (\ref{eq:TransRule}) as $J(k ,z )=1$ for all $k\in K$ and $z\in D$.  As
$K(z,0)=K(0,w)=1$ for all $z,w\in D$ we get from the transformation rule (\ref{eq:TransRule})
that $J(x,0)J(x,x^{-1}\cdot 0) =1= J(x^{-1},0)h(x^{-1},x\cdot 0)$. Together with a) this implies that $J (x,x^{-1}\cdot 0)=h(x)^{-1}=J(x^{-1},x\cdot 0)$ and
the claim follows.
\smallskip
 
 \noindent
f) The function $J(g,z)$ is the complex  Jacobian of the $G$ action at the point $z\in D$. The $G$-action extends to a smooth action  on
$\bar D$, 
the closure of $D$,  which is compact. Thus $J(g,z)$ is well defined and smooth on $G\times \bar D$. As $U^{-1} \times \bar D $ is compact
and $J(g,z)$ is never zero on $G\times \bar D$, it follows that there exist constants
$0<C_1<C_2$ such that $C_1\le |J(u^{-1},z)|\le C_2$ for $u\in U$ and $z\in D$. We also have
\[h(xu)=h(u^{-1}x^{-1}) = |J(u^{-1}x^{-1},0)|^{2/g}=|J(u^{-1},x^{-1}\cdot 0)|^{2/g}h(x^{-1}) \]
The claim now follows as $h(x^{-1}) = h(x)$.
  \end{proof}

\begin{remark} The statement f) also follows from \cite[Theorem 3.8]{Faraut1990} and the following lemma.
\end{remark}

\begin{lemma}\label{le:hxhy}  Let $z=x\cdot o, w=y\cdot o\in D$. Then
\[h(y^{-1}x)=|J(y^{-1}x,o)|^{2/g}=\frac{h(z)h(w)}{|h(z,w)|^2}\, .\]
\end{lemma}
\begin{proof} We first note that $h(u,0)=h(0,v)=1$ for all $u,v\in D$. Hence by Lemma
\ref{lem:2.1} part b:
\begin{align*}
1&=h(y^{-1}x\cdot o, y^{-1}y\cdot o)\\
&=J(y^{-1},x\cdot o)^{1/g}\overline{J(y^{-1},y\cdot o)}^{1/g}h(z,w)\\
&=|J(y^{-1},x\cdot o)|^{1/g}|J(y^{-1},y\cdot o)|^{1/g}|h(z,w)|\, .
\end{align*}

By part (c) in the above lemma we have
\[|J(y^{-1},x\cdot o)|^{2 / g} =|J(y^{-1}x,o)| |J(x,o)|^{-1}=h(y^{-1}x)/h(x)\, .\]
Then take $x=y$ to get
\[|J(y^{-1},y\cdot o)|^{2/g}=1/h(y)\, .\]
This proves the statement.
 \end{proof}

Next we present generalized Forelli-Rudin
estimates for the Bergman kernel due to
Faraut and Koranyi (see Theorem 4.1, p. 80, in \cite{Faraut1990}). Denote by $\partial D$ the boundary of the bounded
domain $D$.
\begin{theorem}
  \label{thm:FK}
For $b\in \R$ and $c>g-1$ define 
\[J_{b,c}(z)=\int_D \frac{h(w,w)^{c-g}}{|h(z,w)|^{b+c}}\, dw\, ,\quad z\in D\, .\]
Then the following holds:
\begin{enumerate}
\item $J_{b,c}(z)$ is bounded on $D$ if and only if $b<-(r-1)a/2$.
\item If $b>(r-1)a/2$ then
  \[J_{b,c}(z)\sim h(z,z)^{-b}, \quad \text{as } z\to \partial D\, .\]
\end{enumerate}
\end{theorem}
These estimates provide the following
result about boundedness for an integral operator with kernel given
by the modulus of the Bergman kernel.

\begin{theorem}\label{thm:bekollekagou}
For $p>1$ the operator $T$ defined by
$$
Tf(z) = \int_D f(w) |h(z,w)|^{-\gamma} h(w)^{\gamma-g}\,dw
$$  
is bounded $L^p_\alpha(D)\to L^p_\alpha(D)$ if $\gamma>g-1+(r-1)a/2$ and
$$
g-1 - (r-1)a/2 +p(r-1)a/2 < \alpha < g-1-(r-1)a/2 + p(\gamma+1-g).
$$
\end{theorem}
A more general result was proved by B\'ekoll\'e and Temgoua Kagou
(see Theorem II.7 in \cite{Bekolle1995}).

\begin{proof}
  This involves a standard trick based on Schur's lemma.
  Notice that for $\epsilon >0$ we have, by H\"older's inequality with
  $1/p+1/q=1$,
  \begin{align*}
    \| T_\gamma f\|_{L^p_\alpha(D)}^p
    &\leq \int_D \left(
      \int_D |f(w)| h(w)^\epsilon|h(z,w)|^{-\gamma} h(w)^{-\epsilon}
      h(w)^{\gamma-g}\,dw 
      \right)^p h(z)^{\alpha-g}\,dz \\
    &\leq
      \int_D \left(
      \int_D |f(w)|^p h(w)^{p\epsilon}|h(z,w)|^{-\gamma}h(w)^{\gamma-g}\,dw 
      \right) \\
    &\quad \times
      \left( \int_D h(w)^{-q\epsilon}|h(z,w)|^{-\gamma}h(w)^{\gamma-g}\,dw 
      \right)^{p/q}
      h(z)^{\alpha-g}\,dz \\
    \intertext{which by Theorem~\ref{thm:FK} is, up to a constant, less than}
    &\leq
      \int_D \left(
      \int_D |f(w)|^p h(w)^{p\epsilon}|h(z,w)|^{-\gamma}h(w)^{\gamma-g}\,dw 
      \right) \left( h(z)^{-q\epsilon} \right)^{p/q}
      h(z)^{\alpha-g}\,dz,
      \intertext{if $q\epsilon>(r-1)a/2$. By Tonelli's theorem this equals}
    &= \int_D |f(w)|^p h(w)^{p\epsilon} h(w)^{\gamma-g} \int_D 
      |h(z,w)|^{-\gamma} h(z)^{-p\epsilon} h(z)^{\alpha-g}\,dz \,dw ,
      \intertext{which by Theorem~\ref{thm:FK} is less than}
    &\leq \int_D |f(w)|^p h(w)^{p\epsilon} h(w)^{\gamma-g}
      h(w)^{-\gamma+\alpha-p\epsilon}\,dw
      \intertext{if $\gamma+p\epsilon-\alpha>(r-1)a/2.$ Finally this
      equals}
    &= \| f\|_{L^p_\alpha(D)}^p.
  \end{align*}
  The restrictions for the parameters $\alpha$ and $\gamma$ required by
  Theorem~\ref{thm:FK} can be rewritten to
  $$
  g-1 - (r-1)a/2 +p(r-1)a/2 < \alpha < g-1-(r-1)a/2 + p(\gamma+1-g).
  $$
  
\end{proof}

The proof of the wavelet characterization in Theorem~\ref{thm:waveletcharbergman} will be carried out using ideas from
coorbit theory for projective representations as presented in \cite{Christensen2019}.
For this purpose we will rephrase Theorem~\ref{thm:bekollekagou}
in the form of boundedness of convolution operators on $G$.
 
  For the remainder of this paper we normalize the Haar
  measure on $G$
  such that if $f$ is $K$-right-invariant and $\widetilde{f}(z) = f(x)$
  when $z=xK$, then
  $$
  \int_G f(x)\,dx = \int_{D} \widetilde{f}(z) h(z)^{-g}\,dz.
  $$

For functions $F,G$ on the group $G$ we define convolution in
the usual way:
$$
F*G(x) = \int F(y)G(y^{-1}x)\,dy.
$$

\begin{theorem}
  \label{thm:convoperator}
  The operator $C_\gamma:L^p_{\alpha-\gamma p/2}(G)\to L^p_{\alpha-\gamma p/2}(G)$
  defined by
  $$
  C_\gamma F(x)  =  \int_G F(y) h(y^{-1}x)^{\gamma/2} \,dy
 = F*h^{\gamma/2}(x)
  $$
  is bounded if  $
  g-1 +(p-1)(r-1)a/2 < \alpha < g-1 + p(\gamma+1-g)- (r-1)a/2.
  $
\end{theorem}

\begin{proof}
  We already know that $h(y^{-1}x) = h(x)h(y)/|h(x,y)|^2$, so $C_\gamma$ takes the form
  \begin{align*}
  C_\gamma F(x) 
  &= h(x)^{\gamma/2}  
  \int_G F(y)h(y)^{\gamma/2}|h(x,y)|^{-\gamma} \,dy \\
  &= h(x)^{\gamma/2}    \int_D \widetilde{F}(w)|h(x\cdot o,w)|^{-\gamma} h(w)^{\gamma-g} \,dw
  \end{align*}
  where 
  $$
  \widetilde{F}(w) = \frac{1}{h(w)^{\gamma/2}}\int_K F(yk)\,dk
  $$
  for $w=y\cdot o$.
  By H\"older's inequality
  $$
  \int_K |F(yk)| \,dk \leq \left(\int_K |F(yk)|^p\,dk \right)^{1/p} \left( \int_K 1\,dk\right)^{1/q}
  = \left(\int_K |F(yk)|^p\,dk \right)^{1/p},
  $$
  so by Fubini's Theorem and the right-$K$-invariance of $h$ we get
  \begin{align*}
   \|\widetilde{F}\|_{L^p_\alpha(D)}^p
   &= \int_D |\widetilde{F}(w)|^p h(w)^{\alpha-g} \,dw \\
   &= \int_D \left| \int_K F(yk)\,dk \right|^p h(w)^{\alpha-\gamma p/2-g} \,dw \\
   &= \int_G \left| \int_K F(yk)\,dk \right|^p h(y)^{\alpha-\gamma p/2} \,dy \\
   &\leq \int_G\left( \int_K |F(yk)|^p \,dk\right) h(y)^{\alpha-\gamma p/2} \,dy \\
   &\leq \int_G |F(y)|^p h(y)^{\alpha-\gamma p/2} \,dy.
  \end{align*}

  Since $C_\gamma F$ is right-$K$-invariant its $L^p_{\alpha-\gamma p/2}(G)$-norm can be
  estimated by
  \begin{align*}
    \int_G |C_\gamma F(x)|^p h(x)^{\alpha-\gamma p/2}\,dx
    &= \int_G \left| \int_D \widetilde{F}(w)     |h(x\cdot o,w)|^{-\gamma} h(w)^{\gamma-g} \,dw \right|^p h(x)^\alpha \,dx \\
    &= \int_D \left| \int_D \widetilde{F}(w)     |h(z,w)|^{-\gamma} h(w)^{\gamma-g} \,dw \right|^p h(z)^{\alpha-g} \,dz \\
    &= \int_D \left| T \widetilde{F}(z)  \right|^p h(z)^{\alpha-g} \,dz \\
    &= \| T\widetilde{F}\|_{L^p_\alpha(D)}^p \\
    &\leq C \|\widetilde{F}\|_{L^p_\alpha(D)}^p \\
    &\leq C \|F\|_{L^p_{\alpha-\gamma p/2}(G)}^p
  \end{align*}
  In the second to last estimate we used the boundedness of the operator $T$.
  This shows that $C_\gamma$ is bounded on $L^P_{\alpha-\gamma p/2}(G)$ for the specified ranges of $\alpha$,$\gamma$ and $p$.
  \end{proof}

\section{Proof of the first part of Theorem~\ref{thm:waveletcharbergman}}

We now turn to the proof of Theorem~\ref{thm:waveletcharbergman}.
\begin{lemma}\label{lem:alphabeta}
  If $\alpha>g-1$ then for $\beta > 2\alpha/p+g-1$ we have
  $A^p_\alpha \subseteq A^2_\beta$.
\end{lemma}
\begin{proof}
  Lemma 2.1 in \cite{Rochberg1982} states that
  $$
  |f(z)| \leq C h(z)^{-\alpha/p} \| f\|_{A^p_\alpha}.
  $$
  Therefore,
  $$
  \| f\|_{A^2_\beta}^2 \leq
  C \| f\|_{A^p_\alpha}^2 \int_D h(z)^{\beta-2\alpha/p-g}\,dz 
  $$
  which is finite if $\beta-2\alpha/p > g-1$.
\end{proof}

\begin{lemma}\label{lem:bergmanindual}
  If $\alpha,\gamma >g-1$ then $A^p_\alpha \subseteq S_\gamma^*$.
\end{lemma}
\begin{proof}
  First, if $\alpha>g-1$, then by Lemma~\ref{lem:alphabeta},
  there is a $\beta>2\alpha/p+g-1$ for which $A^p_\alpha \subseteq A^2_\beta$.
  It is known that the monomials $\psi_m(z)=\frac{1}{\|z^m\|_{A^2_\beta}} z^m$, $m=(m_1,\ldots ,m_n)$ and
  $z^m = z_1^{m_1}\cdots z_n^{m_n}$,
  form an orthonormal basis for $A^2_\beta$, so
  $$ f = \sum_{m\geq 0}
   \langle f, \psi_m \rangle_{\beta} \psi_m.$$
  Let $P_k$ denote the space of polynomials of homogeneous degree $k$,
  then $f=\sum_k f_k$ for $f_k\in P_k$, and $f_k$ is given by
  $$ f_k = \sum_{|m|=k} \langle f, \psi_m \rangle_{\beta} \psi_m.$$
  Taking the $A^2_\gamma$ norm and using the triangle inequality, we get
  \begin{equation*}
    \| f_k\|_{A^2_\gamma}
    \leq \sum_{|m|=k} 
      \| \psi_m\|_{A^2_\gamma} |\langle f, \psi_m \rangle_{\beta}| 
    \leq
    \sum_{|m|=k} \frac{\| \psi_m\|_{A^2_\gamma}}{\|\psi_m\|_{A^2_\beta}}
    \| f\|_{A^2_\beta}.
  \end{equation*}
  By Proposition 3.3 in \cite{Chebli2004} there are integers
  $N_1$,$N_2$ and a constant $C$, which only depend on
  $\beta$ and $\gamma$,
  such that 
  $$
  \frac{\| \psi_m\|_{A^2_\gamma}}{\|\psi_m\|_{A^2_\beta}} \leq C(1+k)^{N_1-N_2}.
  $$
  From this estimate we get
  \begin{align*}
    \| f_k\|_{A^2_\gamma}
    &\leq  C\| f\|_{A^2_\beta} \sum_{|m|=k} (1+k)^{N_1-N_2} \\
    &\leq  C\| f\|_{A^2_\beta} \dim(P_k) (1+k)^{N_1-N_2}.
  \end{align*}
  The fact that there is a constant $A$ and an integer $N_3$ for which
  $$
  \dim(P_k) \leq A(1+k)^{N_3}
  $$
  finishes the proof that $\| f_k\|_\gamma \leq C (1+k)^N$ for some integer
  $N$ and for some constant $C$ which do not depend on $k$.
  Therefore, $f$ is in $S_\gamma^*$.
\end{proof}

To prove the first part of Theorem~\ref{thm:waveletcharbergman}
let $\psi=1$ and consider the representation
$\pi_\gamma$. 
Then the wavelet coefficients with $\psi=1$ and $f\in S_\gamma^*$
become
\begin{align}
  W^\gamma_\psi(f)(x) 
  &= \int_D f(z)\overline{J(x^{-1},z)}^{\gamma/g} h(z)^{\gamma-g} \,dz \notag\\
  &= \int_D f(z) K_\gamma(x\cdot o,z) J(x,o)^{\gamma/g} h(z)^{\gamma-g} \,dz \notag\\
  &= J(x,o)^{\gamma/g} f(x\cdot o).
    \label{eq:waveletforconstant}
  \end{align}
  We see that
  \begin{equation}
    \label{eq:isometry}
    \int_G |W^\gamma_\psi(f)(x)|^p h(x)^{\alpha-\gamma p/2} \,dx
    = \int_G |f(x\cdot o)|^p h(x)^{\alpha} \,dx
    = \int_D |f(z)|^p h(z)^{\alpha-g} \,dz,
  \end{equation}
    and from this it follows that if $W^\gamma_\psi(f)$ is in
  $L^p_{\alpha-\gamma p/2}(G)$ then $f\in A^p_\alpha$.
  If $f\in A^p_\alpha$, then by Lemma~\ref{lem:bergmanindual}
  $f$ is in $S_\gamma^*$. Then equation \eqref{eq:isometry} tells us that
  $W^\gamma_\psi(f)$ is in $L^p_{\alpha-\gamma p/2}(G)$.
  This finishes the proof of the first part of
  Theorem~\ref{thm:waveletcharbergman}.

\section{Proof of the second part of Theorem~\ref{thm:waveletcharbergman}}

Before we can prove the second part of the theorem, we need to better understand the
growth of wavelet coefficients for vectors in $S_\gamma$.
The following gives the estimates of wavelet coefficients which
turn out to be a crucial part of our proof.

\begin{proposition}
  \label{matrixdecay}
  If $u,v\in S_\gamma$ and $x\in G$,
  then
  $$
  |W^\gamma_u(v)(x)| \leq C h^{\gamma/2}*h^{\gamma/2}(x).
  $$
  Moreover, if either $u$ or $v$ is a polynomial, then
  $$
  |W^\gamma_u(v)(x)| \leq C h^{\gamma/2}(x).
  $$
\end{proposition}
\begin{remark}
  Using this result and Theorem~\ref{thm:FK} it is possible to show that
  $|W_u(v)| \leq C h^{\gamma/2 - \epsilon}$ for any $\epsilon>(r-1)a/2$.
  We expect that this inequality can be verified for all $\epsilon>0$.
  This would enable us to answer some of the questions
  posed in Remark~\ref{rem:waveletchar}.
\end{remark}
\begin{proof}
  We first assume that $u\in S_\gamma$ and $v$ is a polynomial.
  Let us calculate the wavelet coefficient for $x\in\widetilde{G}$
  \begin{align*}
    W^\gamma_u(v)(x)
    &= (v,\rho_\gamma(x)u) \\
    &= \int_D v(z)
      \overline{J(x^{-1},z)^{\gamma/g} u(x^{-1}\cdot z) } \,d\mu_\gamma(z) \\
    &= J(x,o)^{\gamma/g}
      \int_D v(z) \overline{K(z,x\cdot o)^{\gamma/g} u(x^{-1}\cdot z) } \,d\mu_\gamma(z) \\
    &= J(x,o)^{\gamma/g}
      \int_D v(z) \overline{K_\gamma(z,x\cdot o) u(x^{-1}\cdot z) } \,d\mu_\gamma(z).
  \end{align*}
  Since $K_\gamma$ is the reproducing kernel for $A^2_\gamma$ and
  since $\psi_m = \frac{1}{\| z^m\|_{A^2_\gamma}}z^m$ is an
  orthonormal basis for $A^2_\gamma$, we know that
  $$
  K_\gamma(z,w) = \sum_{m\geq 0} \psi_m(z)\overline{\psi_m(w)}.
  $$
  This yields
  \begin{equation*}
    W^\gamma_u(v)(x)
    = J(x,o)^{\gamma/g}
      \sum_{m\geq 0} \psi_m(x\cdot o)
      \int_D v(z) \overline{\psi_m(z) u(x^{-1}\cdot z) } \,d\mu_\gamma(z).
  \end{equation*}
  Since $v$ is a polynomial  
    and $z\mapsto u(x^{-1}\cdot z)$ is holomorphic,
  we get that the sum over $m$ is finite, i.e.
  \begin{equation*}
    W^\gamma_u(v)(x)
    = J(x,o)^{\gamma/g}
      \sum_{m\in M} \psi_m(x\cdot o)
      \int_D v(z) \overline{\psi_m(z) u(x^{-1}\cdot z) } \,d\mu_\gamma(z).
  \end{equation*}
  where $M$ is a finite index set  (which can be chosen independently of $x$).
  By Lemma 1.3 and Proposition 3.3 in \cite{Chebli2004} we
  get that $u$ is bounded. Therefore, 
  there is a constant $C_m$ for which
  $$
  \int_D |v(z) \overline{\psi_m(z) u(x^{-1}\cdot z) }| \,d\mu_\gamma(z)
  \leq C_m.
  $$
  By the finiteness of $M$ we can therefore derive that for
  $x\in G$ we have
  $$
  |W^\gamma_u(v)(x)| \leq C |J(x,o)|^{\gamma/g} = Ch^{\gamma/2}(x).
  $$

  Now assume that $v\in S_\gamma$
  is not a polynomial.
  For $x,y\in\widetilde{G}$ 
  equation
  (\ref{eq:waveletforconstant}) gives
  $$
  W_1(v)(y)W_u(1)(y^{-1}x)
  = h(y)^{\gamma} v(y\cdot o)\overline{\pi_\gamma(x)u(y\cdot o)}.
  $$
  This function is well-defined for $y$ in $G$, and
  $$
  \int_G W_1(v)(y)W_u(1)(y^{-1}x) \,dy
  = \int_D v(z)\overline{\pi_\gamma(x)u(z)} h(z)^{\gamma-g}\,dz
  = W_u^\gamma(v)(x).
  $$
  From this we get
  $$
  |W^\gamma_u(v)|
  \leq  |W^\gamma_1(v)|*|W^\gamma_u(1)|.
  $$
  By unitarity $|W_1^\gamma(v)(x)|=|W_v^\gamma(1)(x^{-1})|
  \leq Ch^{\gamma/2}(x^{-1}) = Ch^{\gamma/2}(x)$
  which concludes the proof.
\end{proof}

Recall that a Banach space  $B$ of functions on $G$ is said to be solid if $f\in B$, $f\ge 0$, and $g$ is measurable function
on $G$ with $|g|\le f$ implies that $g\in B$. A typical examples are the spaces $L^p(G)$, $1\le p\le \infty$.

\begin{corollary}
  \label{cor:boundedconv}
  If $f\mapsto f*h^{\gamma/2}$
  is bounded on a solid Banach function
  space $B$ on $G$, then $f\mapsto f*|W^\gamma_u(v)|$ is bounded for
  all $u,v$ in $S_\gamma$.
\end{corollary}
\begin{proof}
  From the previous result we get
  $$
  | f*|W^\gamma_u(v)| |
  \leq C |f|*h^{\gamma/2}*h^{\gamma/2}.
  $$
  Since $B$ is solid, $|f|$ is in $B$ if $f$ is, and
  then $|f|*h^{\gamma/2}*h^{\gamma/2}$ is in $B$ by assumption.
  By solidity $f*|W^\gamma_u(v)|$ is in $B$ and
  $\| f*|W^\gamma_u(v)|\|_B \leq C \|f\|_B$.
\end{proof}

So far we have not needed to introduce a projective representation of
$G$ along with its cocycle. We choose to introduce them now,
since subsequent arguments are easily
carried out based on established knowledge
about square integrable  projective representations and twisted convolution.
In particular we will prove the second part of
Theorem~\ref{thm:waveletcharbergman}
by using twisted convolution to swap the vector $\psi=1$ in
$W_\psi^\gamma(f)$ by an arbitrary vector in $S_\gamma$.

Let $\rho_\gamma$ be a projective representation of $G$
corresponding to the representation $\pi_\gamma$ of $\widetilde{G}$,
and let $\sigma_\gamma$ be the corresponding cocycle for $\rho_\gamma$.
We retain the notation $W_\psi^\gamma:S^*_\gamma\to \mathcal{M}(G)$
for the wavelet transform
$$
W_\psi^\gamma(f)(x) = \langle f,\rho_\gamma(x)\psi\rangle,
$$
and note that it agrees with the previous notation up to a unimodular factor.
Here $\mathcal{M}(G)$ denotes the space
  of Borel measurable functions on $G$.
  Also define twisted convolution of $f,g$ by
  $$
  f\#_\gamma g(x) = \int_G f(y)g(y^{-1}x)
  \overline{\sigma_\gamma(y,y^{-1}x)}\,dy
  $$
  when the integral makes sense.
  \begin{lemma} \label{lem:twistedconv}
      Let $\phi,\psi$ be in $S_\gamma$. If $F\in L^p_{\alpha-\gamma p/2}(G)$,
      then $F\#W^\gamma_\psi(\phi)$ exists as an integrable function
      if     $g-1 + (p-1)(r-1)a/2< \alpha < p(\gamma-g+1)+g-1-(r-1)a/2$.
      Moreover, $F\mapsto F\#W^\gamma_\psi(\phi)$ is bounded
      $L^p_{\alpha-\gamma p/2}(G)\to L^p_{\alpha-\gamma p/2}(G)$ .
  \end{lemma}

  \begin{proof}
    We know from Proposition~\ref{matrixdecay}
    that $|F|*|W^\gamma_\psi(\phi)|\leq C |F|*(h^{\gamma/2}*h^{\gamma/2})$.
    By Tonelli's Theorem
    this equals $(|F|*h^{\gamma/2})*h^{\gamma/2}$.
    From Theorem~\ref{thm:convoperator} we know that
     $|F|*h^{\gamma/2}$ is in $L^p_{\alpha-\gamma p/2}(G)$,
     which means that $|F|*(h^{\gamma/2}*h^{\gamma/2})$
     exists as an integrable function. This also means
     that $|F|*|W^\gamma_\psi(\phi)|$
     exists and therefore so
     does $F\#W^\gamma_\psi(\phi)$.
    Another application of Theorem~\ref{thm:convoperator}
   proves the continuity statement.
   \end{proof}    
    
 For a function $F$ on $G$ and a vector $X\in\mathfrak{g}$
 define the derivative
 $$
 XF(x) = \frac{d}{dt}\Big|_{t=0} F(e^{-tX}x).
 $$
 \begin{lemma}\label{lem:growingsmoothfunction}
   There is a sequence of smooth compactly supported
   functions $\psi_n:G\to [0,1]$ such that
   $\psi_{n+1}\geq \psi_n$ and for any finite collection
   $X_1,X_2,\dots,X_N\in \mathfrak{g}$
   there is a constant $C_N$ such that
   for every $n$ we have
   $\| X_1X_2\dots X_N\psi_n\|_\infty \leq C_N$.
 \end{lemma}

 \begin{proof}
   We can take $\psi_n^\vee(x)=\psi_n(x^{-1})$
   to be partial sums of a partition of unity
   as contstructed on p. 329 in \cite{Triebel1988}.
 \end{proof}
 
 \begin{proposition}\label{prop:swapvector}
     Let $\psi,\phi$ and $\eta$ be in $S_\gamma$
     and assume that
    $$g-1 + (p-1)(r-1)a/2< \alpha < p(\gamma-g+1)+g-1-(r-1)a/2.$$
   If
   $f$ is in $S_\gamma^*$ and $W_\psi^\gamma(f)$ is in $L^p_{\alpha-\gamma p/2}(G)$,
   then $W^\gamma_\psi(f)\#W^\gamma_\phi(\eta)$ equals
   $d_\gamma^{-1} \langle \eta,\psi \rangle_{A^2_\gamma} W^\gamma_\phi(f)$
   where $d_\gamma$ is the formal dimension of $\pi_\gamma$.
\end{proposition}

\begin{proof}
    First, note that if $W^\gamma_\psi(f)$ is in $L^p_{\alpha-\gamma p/2}(G)$, then
    by Lemma~\ref{lem:twistedconv} the twisted convolution is defined in terms of an integrable function.
  Therefore we can employ the Lebesgue Dominated Convergence Theorem to get that
  \begin{equation}\label{eq:twistedconv}
  W^\gamma_\psi(f)\#W^\gamma_\phi(\eta)(x) = \lim_{n\to\infty} \int_G W^\gamma_\psi(f)(y)W^\gamma_\phi(\eta)(y^{-1}x)\psi_n(y)   \overline{\sigma_\gamma(y,y^{-1}x)}\,dy,
  \end{equation}
  where $\psi_n\in C^\infty_c(G)$ is a sequence of functions which are
  equal to one on growing compact sets whose union is $G$ and which satisfy
  $0\leq \psi_n(x)\leq 1$ for all $x\in G$.
  The equation (\ref{eq:twistedconv}) rewrites to
  $$
  W^\gamma_\psi(f)\#W^\gamma_\phi(\psi)(x) 
  = \lim_{n\to\infty} \int_G \psi_n(y)\langle f,\rho_\gamma(y)\psi \rangle \langle \rho_\gamma(y)\eta,\rho_\gamma(x)\phi\rangle\,dy.
  $$
  Replace $\rho_\gamma(x)\phi$ by a general smooth vector
  $\xi$ and 
  define the smooth compactly supported function
  $$
  \Psi_n(y) = \psi_n(y) \overline{\langle \rho_\gamma(y)\eta ,\xi \rangle},
  $$
  then
  $$
  \int_G \psi_n(y)\langle f,\rho_\gamma(y)\psi \rangle \langle \rho_\gamma(y)\eta,\xi\rangle\,dy
  = 
  \int_G \langle f,\Psi_n(y)\rho_\gamma(y)\psi\rangle\, dy.
  $$
  Since $\Psi_n$ is smooth and compactly supported,   by
  Theorem 3.27 in \cite{Rudin1991} the vector
  $$
  \rho_\gamma(\Psi_n)\psi := \int_G \Psi_n(y)\rho_\gamma(y)\psi\, dy
  $$
  is a smooth vector, and
  $$
\int_G \langle f,\Psi_n(y)\rho_\gamma(y)\psi\rangle\, dy
 = \langle f, \rho_\gamma(\Psi_n)\psi \rangle.
 $$
 In order to finish the proof it suffices to show that the
 latter expresson converges to
 $c\langle f,\xi\rangle$
 where $c=d^{-1}_\gamma \langle \psi,\eta \rangle$.
 We will do so by demonstrating that 
 the vectors $\rho_\gamma(\Psi_n)\psi$ converge to
 $c \xi$
 in $S_\gamma$.

 Let us first verify convergence in $A^2_\gamma$.
 $$
 \| \rho_\gamma(\Psi_n)\psi - c\xi  \|^2_{A^2_\gamma}
 =
 \| \rho_\gamma(\Psi_n)\psi\|^2_{A^2_\gamma} + |c|^2\|  \xi\|^2_{A^2_\gamma}
 - (\overline{c}\langle \rho_\gamma(\Psi_n)\psi,  \xi  \rangle
 - c\langle \xi  , \rho_\gamma(\Psi_n)\psi \rangle).
 $$
 By the Lebesgue dominated convergence theorem and square integrability 
 of the representation, and the unimodularity of the group, we have that
 $$
    \lim_{n\to\infty}\langle \rho_\gamma(\Psi_n)\psi,  \xi   \rangle
    = \lim_{n\to\infty} \int \psi_n(y)
    \langle \xi,\rho_\gamma(y)\eta \rangle \langle \rho_\gamma(y)\psi,\xi \rangle \, dy
    = \int \langle \xi,\rho_\gamma(y)\eta \rangle \langle \rho_\gamma(y)\psi,\xi \rangle \, dy
    = c \| \xi\|^2_{A^2_\gamma}.
    $$
    Therefore we just need to check that
    $$
    \lim_{n\to\infty}
    \| \rho_\gamma(\Psi_n)\psi\|_{A^2_\gamma}^2 =  |c|^2\|  \xi\|^2_{A^2_\gamma}.
    $$
    Notice that by Fubini we have
    $$
    \| \rho_\gamma(\Psi_n)\psi\|^2_{A^2_\gamma}
    = \int_G \int_G \psi_n(x)\psi_n(y)
    \langle \xi,\rho_\gamma(x)\eta\rangle
    \langle \rho_\gamma(x)\psi ,\rho_\gamma(y)\psi\rangle
    \langle \rho_\gamma(y)\eta,\xi \rangle \, dx\, dy,
    $$
    and we will be able to apply the Lebesgue dominated convergence theorem if
    we can show that the function
    $|\langle \xi,\rho_\gamma(x)\eta\rangle
    \langle \rho_\gamma(x)\psi ,\rho_\gamma(y)\psi\rangle
    \langle \rho_\gamma(y)\eta,\xi \rangle|$ is integrable on $G\times G$.
    This is the same as showing that the integral
    $$
    \int_G \int_G |W^\gamma_\eta(\xi)(y)||W^\gamma_\psi(\psi)(y^{-1}x)||W^\gamma_\eta(\xi)(x)| \,dx\,dy
    $$
    is finite.
    From Proposition~\ref{matrixdecay}
    it is known that $W^\gamma_\eta(\xi)\in L^2(G)$ and 
    $W^\gamma_\psi(\psi)\in L^p(G)$
    for some $1<p<2$, so the Kunze-Stein phenomenon
    \cite{Cowling1978} tells us that the integral is finite since
    $|W^\gamma_\eta(\xi)|*|W^\gamma_\psi(\psi)|$ is again in $L^2(G)$.
    Therefore
    $$
      \lim_{n\to\infty}
      \| \rho_\gamma(\Psi_n)\psi\|_{A^2_\gamma}^2
    = \int_G \int_G \langle \xi,\rho_\gamma(x)\eta\rangle
    \langle \rho_\gamma(x)\psi ,\rho_\gamma(y)\psi\rangle
    \langle \rho_\gamma(y)\eta,\xi \rangle \, dx\, dy
    =  |c|^2\|  \xi\|^2_{A^2_\gamma}
    $$
    by the orthogonality relations \cite[Theorem 3]{Aniello2006}.
   
    We can now repeat the argument with derivatives of the vector
    $\rho_\gamma(\Psi_n)\psi$ to show it converges to derivatives of
    $c\eta$.
    If $X\in\mathfrak{g}$ then
    define $\rho_\gamma(X) = \lim_{t\to 0}\frac{1}{t} (\rho_\gamma(e^{tX})-I).$
    We have
    $$
    \rho_\gamma(X)\rho_\gamma(\Psi_n)\psi
    = W_\psi^\gamma (\eta)(y) X\psi_n(y)  +
     W_\psi^\gamma(\rho_\gamma(X)\eta)\psi_n(y).
    $$
    As before, it follows from the Lebesgue dominated
    convergence theorem and the assumption that $X\psi_n$ is
    uniformly bounded in $n$, that
    $\rho_\gamma(X)\rho_\gamma(\Psi_n)\psi$ converges to $c \rho_\gamma(X)\eta$
    in $A^2_\gamma$. This argument can be repeated to show that
    $\rho_\gamma(\Psi_n)\psi$ converges to $c\eta$ in $S_\gamma$.
\end{proof}

\begin{remark}
  If one chooses to define the twisted convolution of $W_\psi(f)$ and $W_\phi(\eta)$ 
  using the expression \eqref{eq:twistedconv}, then the proof in fact shows
  that it equals $d_\gamma^{-1} \langle \eta,\psi \rangle_{A^2_\gamma} W^\gamma_\phi(f)$
  regardless of the integrability condition.
  This could be a key observation for applying coorbit theory in this setting for dealing
  with the extended range of $\alpha$.
\end{remark}

We have gathered all the results needed to finish the proof of Theorem~\ref{thm:waveletcharbergman}.
First, we know that $f\in A^p_\alpha$ if and only if $f\in S_\gamma^*$ and
$W^\gamma_1(f)\in L^p_{\alpha-\gamma p/2}(G)$. Employ
Proposition~\ref{prop:swapvector}
to show that for an arbitrary vector $\psi\in S_\gamma$ we have
$W^\gamma_\psi(f) = C W_1(f)\#W_\psi(1)$ for non-zero constant $C$.
By Lemma~\ref{lem:twistedconv} we get that
$$
\| W^\gamma_\psi(f)\|_{L^p_{\alpha-\gamma p/2}(G)} = C \| W_1(f)\#W_\psi(1)\|_{L^p_{\alpha-\gamma p/2}(G)}
\leq C' \| W_1(f)\|_{L^p_{\alpha-\gamma p/2}(G)}.
$$
Therefore $W^\gamma_\psi(f)$ is in $L^p_{\alpha-\gamma p/2}(G)$.

If, on the other hand, $W^\gamma_\psi(f)$ is in $L^p_{\alpha-\gamma p/2}(G)$, then
Lemma~\ref{lem:twistedconv} and Proposition~\ref{prop:swapvector} tell us that
$$\| W^\gamma_1(f)\|_{L^p_{\alpha-\gamma p/2}(G) } 
= D \| W_\psi(f)\#W_1(\psi)\|_{L^p_{\alpha-\gamma p/2}(G) }
\leq D' \| W_\psi(f)\|_{L^p_{\alpha-\gamma p/2}(G) },$$
where $D$ is a non-zero constant.
By the first part of Theorem~\ref{thm:waveletcharbergman} we
get that $f$ is in $A^p_\alpha$.
These calculations also show that the norms are equivalent, and the proof is done.

\section{Proof of Theorem~\ref{thm:atomicdecompbergman}}
Normalize $\psi\in S_\gamma$ such that
$W^\gamma_\psi(\psi)\# W^\gamma_\psi(\psi) = W^\gamma_\psi(\psi)$.
Then $L^{p}_{\alpha-\gamma p/2}(G)\# W^\gamma_\psi(\psi)$
is a non-zero closed Banach subspace of $L^{p}_{\alpha-\gamma p/2}(G)$.
This follows from Corollary~\ref{cor:boundedconv} which implies
that $F\mapsto F\# W^\gamma_\psi(\psi)$ is a bounded projection on
$L^{p}_{\alpha-\gamma p/2}(G)$.
By Theorem~\ref{thm:waveletcharbergman} 
the space $L^{p}_{\alpha-\gamma p/2}(G)\# W^\gamma_\psi(\psi)$
is isomorphic to $A^p_\alpha(D)$ via the wavelet transform.
Therefore, if $f\in A^p_\alpha(D)$ we have the following
integral representation of $W_\psi^\gamma(f)$:
$$
W_\psi^\gamma(f) = W_\psi^\gamma(f)\# W^\gamma_\psi(\psi).
$$

The atomic decomposition in Theorem~\ref{thm:atomicdecompbergman} will
follow from a discretization of this integral representation.
This approach is standard in coorbit theory
for integrable representations \cite{Feichtinger1988,Feichtinger1989}
and has been extended to non-integrable projective
representations via estimates
involving the smoothness of the kernel $W^\gamma_\psi(\psi)$
in \cite{Christensen2019}. We refer to these papers for details.

The first thing we note is that the space $S_\gamma$ is the space of
smooth vectors for the representation $\pi_\gamma$ (see \cite{Chebli2004}).
Therefore the space $S_\gamma$ is invariant under the differential operator
$$
\pi_\gamma(X)u = \frac{d}{dt}\Big|_{t=0} \pi_\gamma(e^{-tX})u
$$
where $X$ is in the Lie algebra $\mathfrak{g}$ of $G$.
This means that left and right derivatives
(as defined in \cite{Christensen2019}) of a wavelet coefficient $W^\gamma_u(v)$,
where $u$ and $v$ are in $S_\gamma$,
will correspond to wavelet coefficients $W^\gamma_{u'}(v')$ where $u',v'\in S_\gamma$.
Using Theorem~\ref{thm:convoperator} and
Corollary~\ref{cor:boundedconv}, while noticing that
$|W^\gamma_1(1)| = h^{\gamma/2}$, we
realize that
$L^p_{\alpha-\gamma p/2}(G)*|W^\gamma_{u'}(v')| \subseteq L^p_{\alpha-\gamma p/2}(G)$
for the specified range of parameters when $u'$ and $v'$ are in $S_\gamma$.
This implies that  the solid Banach function space $L^p_{\alpha-\gamma p/2}(G)$
and the vector $\psi\in S_\gamma$ satisfy
\cite[Assumption 3]{Christensen2019},
and then \cite[Theorem 8]{Christensen2019} can be applied as described below.

Let $X_1,\dots,X_n$ be a fixed basis for the Lie algebra $\mathfrak{g}$ of $G$.
For $\epsilon>0$ define the compact neighborhood $U_\epsilon$ of the identity by
$$
U_\epsilon = \{ \exp(t_1X_1)\dots\exp(t_nX_n) \mid -\epsilon \leq t_k
\leq \epsilon \text{ for all } k=1,\dots, n\}.
$$
We will choose the cocycle $\sigma$ and an $\epsilon$ 
small enough such that $\sigma$ is $C^\infty$ on
$U_\epsilon\times U_\epsilon$ (this is always possible due
to \cite[Lemma 7.20]{Varadarajan1985}).
Assume that $x_i$ in $G$ satisfy that $x_iU_\epsilon$ cover $G$,
and that there is a compact neighborhood $V_\epsilon\subseteq U_\epsilon$
such that the $x_iV_\epsilon$ are pairwise disjoint.
Let $\{ \psi_i\}$ be a bounded uniform partition of unity satisfying that
(i) $0\leq \psi_i\leq 1$, (ii) $\sum_i \psi_i = 1$ and (iii) $\mathrm{supp}(\psi_i)\subseteq x_iU_\epsilon$. The existence of the points
  $x_i$ and partitions of unity was proved in \cite{Feichtinger1981}.
Define the operator
$S:L^{p}_{\alpha-\gamma p/2}(G)\# W^\gamma_\psi(\psi)\to L^{p}_{\alpha-\gamma p/2}(G)\# W^\gamma_\psi(\psi)$
by
$$
SF = \sum_i \widetilde{\lambda}_i(F)\ell_{x_i}^{\sigma_\gamma} W^\gamma_\psi(\psi),
$$
where the functionals $\widetilde{\lambda}_i$ are given by
$$
\widetilde{\lambda}_i(F) = \int F(y)\psi_i(y) \overline{\sigma_{\gamma}(y,y^{-1}x_i)}\,dy.
$$
This operator is well-defined, and 
it is possible to choose $\epsilon$ small enough 
that it is invertible. In that
case $F$ in $L^{p}_{\alpha-\gamma p/2}(G)\# W^\gamma_\psi(\psi)$ can
be reconstructed by
$$F=\sum_i \lambda_i(S^{-1}F) \ell_{x_i}^{\sigma_\gamma} W^\gamma_\psi(\psi).$$
Notice, that since $C_c(G)$ are dense in $L^p_{\alpha -\gamma p/2}(G)$ for
$1\leq p <\infty$, this sum converges in norm (see Theorem 8 in \cite{Christensen2019}).
Since $W_\psi^\gamma:A^p_\alpha(D)\to L^{p}_{\alpha-\gamma p/2}(G)\# W^\gamma_\psi(\psi)$ is an
isomorphism we get that
$$
W_\psi^\gamma (f)
= \sum_i \widetilde{\lambda}_i(S^{-1}W_\psi^\gamma(f)) \ell_{x_i}^{\sigma_\gamma} W^\gamma_\psi(\psi)
= W_\psi^\gamma \left( \sum_i \widetilde{\lambda}_i(S^{-1}W_\psi^\gamma(f)) \rho_\gamma (x_i) \psi \right).
$$
for $f\in A_\alpha^p(D)$.
This proves the decomposition
$$
f
= \sum_i \lambda_i(f) \rho_\gamma (x_i) \psi
$$
when the functionals $\lambda_i:A^p_\alpha(D)\to \mathbb{C}$ are defined by
$\lambda_i(f) = \widetilde{\lambda}_i(S^{-1}W_\psi^\gamma(f))$.
According to Theorem 8 in \cite{Christensen2019}
this is an atomic decomposition with sequence space given by the norm
$$\| \{ c_i\}\|
= \left(\int_G \left|\sum_i c_i 1_{x_iU_\epsilon}(x)\right|^p h^{\alpha-\gamma p/2}(x) \, dx\right)^{1/p}.$$
This norm is equivalent to the sequence norm used in Theorem~\ref{thm:atomicdecompbergman}, since there are
constants $0<C_1 < C_2$ such that
$C_1h(x) \leq h(xu)\leq C_2h(x)$ when $u\in U_\epsilon$, see Lemma~\ref{lem:2.1}.

\section{The unbounded case}

In this section we explain how the atomic
decompositions for the unbounded realization of the
domain can be obtained.
The idea is simply to use the Caley transform to transform
information from the bounded domain to the unbounded realization. In particular, the representation theory makes the transition
from the bounded to the unbounded case clear.
Let $c:U\to D$ be the
Cayley transform  from the irreducible unbounded symmetric domain $U$ to the
irreducible bounded symmetric domain $D$. The Cayley transform
  is explicity described in Section 6 in \cite{Koranyi1965}.
Note that there exists
an element $c\in G_\C$, where $G_\C$ is the simply connected complex Lie group with
Lie algebra $\mathfrak{g}_\C$ such that $c(z)=cz$.
For the special case of the unit disc $D=\{z\in\mathbb{C} \mid |z|<1\}$,
the Cayley transform is given by $z\mapsto i\frac{z+1}{-z+1} =
\begin{pmatrix}
  i & i \\ -1 & 1 
\end{pmatrix}\cdot z$.
The Bergman space
$A^2(U)$ is defined as the holomorphic functions for which
$$
\| f\|_{A^2(U)}^2 = \int |f(z)|^2\,dz <\infty.
$$
Here $dz$ denotes the Euclidean measure on $\C^n$ which gives
the unit cube volume 1.
The Bergman space is a reproducing kernel Hilbert space with
reproducing kernel denoted $K_U$ satisfying
$$
f(z) = \int_U f(w)K_U(z,w)\,dz.
$$
Let $J(c,w)$ denote the complex Jacobian of the mapping $c:U\to D$, then
 $C: A^2(D)\to A^2(U)$ defined by $Cf(z)=J(c,z)f(cz)$
is an isometry.
Also, remembering that the Bergman kernel on the bounded realization is
denoted $K$, we have
$$
J(c,z)\overline{J(c,w)} K(cz,cw) = K_U(z,w).
$$
Define the functions $h_U(z,w) = K_U(z,w)^{-1/g}$ and $h_U(z) = h_U(z,z)$,
then we have
$$
h_U(z) = |J(c,z)|^{-2/g} h(cz).
$$
Now define the weighted Bergman space to be the space
of holomorphic functions on $U$ for which
$$
\|f\|_{A^p_\alpha(U)}^p = \int_U |f(z)|^p h_U(z)^{\alpha-g}\,dz < \infty.
$$
Then
$C_{\alpha,p}:A^p_\alpha(D)\to A^p_\alpha(U)$ given
by
$$
C_{\alpha,p} f(z) = J(c,z)^{\frac{2\alpha}{pg}} f(cz)
$$
is an isometry
with inverse
$$
C_{\alpha,p}^{-1}f(w) = J(c^{-1},w)^{\frac{2\alpha}{pg}}f(c^{-1}w).
$$

We can then define a projective representation $\tau_\gamma$ of $G$ on
$A^2_\gamma(U)$ by
$$
\tau_\gamma(x) = C_{\gamma,2} \pi_\gamma(x) C_{\gamma,2}^{-1}.
$$
This projective representation is irreducible, unitary and square integrable,
and the smooth vectors are $T_\gamma = C_{\gamma,2}S_\gamma$ with
dual $T_\gamma^* = (C_{\gamma,2}^{-1})^* S_\gamma^*$ where the
adjoint is defined in terms of the weak dual pairing.
Define the wavelet transform of $f\in T^*_\gamma$ and
$\psi \in T_\gamma$ by
$$
W_\psi^\gamma(f)(x) = \langle f, \tau_\gamma(x)\psi \rangle.
$$
We can now restate our main theorems in the unbounded setting:
\begin{theorem}
  \label{thm:waveletcharunbounded}
  A function $f$ is in the Bergman space $A^p_\alpha(U)$
  if and only if $f\in T_\gamma^*$ and 
  $W_\psi^\gamma (f) \in L^p_{\alpha -\gamma p/2}(G)$
  and either of the following two conditions are satisfied
  \begin{enumerate}
  \item $\gamma,\alpha >g-1$ and
    and $\psi=C_{\gamma,2}1$,
  \item $\gamma > g-1 +(r-1)\frac{a}{2}$ and
    $g-1 - (r-1)\frac{a}{2} + p(r-1)\frac{a}{2} 
    < \alpha 
    < g-1 -(r-1)\frac{a}{2} + p(\gamma-g+1)$ and $\psi \in T_\gamma$ is non-zero.
  \end{enumerate}
  Moreover, the norms $\| f\|_{A^p_\alpha(U)}$ and
  $\| W^\gamma_\psi(f)\|_{L^p_{\alpha-\gamma p/2}}(G)$ are equivalent.
\end{theorem}

Also we have
\begin{theorem}
  \label{thm:atomicdecompunbounded}
  Assume that $\gamma>g-1+(r-1)\frac{a}{2}$ and $\psi\in T_\gamma$ is non-zero.
  If  
  \[g-1 +(p-1)(r-1)\frac{a}{2} 
    < \alpha 
    < g-1 + p(\gamma-g+1)-(r-1)\frac{a}{2},\]
  there is a countable discrete collection of points
  $\{ x_i\}_{i \in I}$ in $G$
  and associated functionals $\{ \lambda_i\}$ on
  $A^p_\alpha(U)$ 
  such that every $f\in A^p_\alpha(U)$ can be written
  $$
  f = \sum_{i\in I} \lambda_i(f) \rho_\gamma(x_i)\psi
  $$
  with $\| \{ \lambda_i(f)\}\|_{\ell^p_{\alpha-\gamma p/2}}
  \leq \| f\|_{L^p_\alpha(U)}$.
  Moreover, if $\{ c_i\} \in \ell^p_{\alpha-\gamma p/2}$ then
  $$
  g := \sum_{i\in I} c_i \rho_\gamma(x_i)\psi
  $$
  is in $A^p_\alpha(U)$ and 
  $\| g\|_{A^p_\alpha(U)} \leq \| \{ \lambda_i(g)\}\|_{\ell^p_{\alpha-\gamma p/2}}$.
\end{theorem}

\begin{remark}
  (1) In the case where $\psi=C_{\gamma,2}1$ we recover the results by
  Coifman and Rochberg with $\alpha=g$ and
  $\gamma=(2\alpha +\theta g)/p$. Namely
  \begin{align*}
  f(z)
  &= \sum_i \lambda_i(f) \overline{J(c^{-1}x_i,o)^{\gamma/g}}
    K_U^{\gamma/g}(z,c^{-1}x_i\cdot o) \\
  &= \sum_i \lambda_i(f) \overline{J(c^{-1}x_i,o)^{\gamma/g}}
    K_U^{\gamma/g}(z,c^{-1}x_i\cdot o).
  \end{align*}
  Thus representation theory makes clear the transition between atomic
    decomopsitions for the two realizations of the domain.

  (2) One could also use the isometry
  $C_{\alpha,p}:A^p_\alpha(D)\to A^p_\alpha(U)$ to transfer the atomic
  decomposition from the bounded domain to the unbounded realization.
  In particular, if $F\in A^p_\alpha(U)$ and if $\psi$ is in $S_\gamma$, then
  $$
  C_{\alpha,p}^{-1}F = \sum_i \lambda_i(C^{-1}_{\alpha,p}F)\pi_\gamma(x_i)\psi.
  $$
  This sum converges absolutely in $A^p_{\alpha}(D)$, and since
  $\pi_\gamma(x_i)\psi\in S_\gamma\subseteq A^p_{\alpha}(D)$ it follows that 
  $$
  F = \sum_i \lambda_i(C^{-1}_{\alpha,p}F)C_{\alpha,p}\pi_\gamma(x_i)\psi.
  $$
  Finally note that this equals
  $$
  F(z) = J(c,z)^{\frac{2\alpha}{pg}-\frac{\gamma}{g}} \sum_i \lambda_i(C^{-1}_{\alpha,p}F)
  \tau_\gamma(x_i)C_{\gamma,2}\psi(z).
  $$
  Due to the factor $J(c,z)^{\frac{2\alpha}{pg}-\frac{\gamma}{g}}$
  the atoms in this decomposition are not translates under $\tau_\gamma$
  of a single function.
  This makes this decomposition obtained in 
  Theorem~\ref{thm:atomicdecompunbounded} more attractive.
\end{remark}


\begin{thebibliography}{10}

\bibitem{Aniello2006}
P.~Aniello.
\newblock Square integrable projective representations and square integrable
  representations modulo a relatively central subgroup.
\newblock {\em Int. J. Geom. Methods Mod. Phys.}, 3(2):233--267, 2006.

\bibitem{Bekolle1986}
D.~B{\'{e}}koll{\'{e}}.
\newblock The dual of the {B}ergman space ${A}^1$ in symmetric {S}iegel domains
  of type $ii$.
\newblock {\em Trans. Amer. Math. Soc.}, 296(2):607--619, 1986.

\bibitem{Bekolle2017}
D.~B\'ekoll\'e, J.~Gonessa, and C.~Nana.
\newblock Atomic decomposition and interpolation via the complex method for
  mixed norm bergman spaces on tube domains over symmetric cones.
\newblock {\em arXiv:1703.07862}. To appear in {\it Ann. Sc. Norm. Super. Pisa Cl. Sci.}

\bibitem{Bekolle1995}
D.~B\'{e}koll\'{e} and A.~Temgoua~Kagou.
\newblock Reproducing properties and {$L^p$}-estimates for {B}ergman
  projections in {S}iegel domains of type {II}.
\newblock {\em Studia Math.}, 115(3):219--239, 1995.

\bibitem{Bruhat1956}
F.~Bruhat.
\newblock Sur les repr\'{e}sentations induites des groupes de {L}ie.
\newblock {\em Bull. Soc. Math. France}, 84:97--205, 1956.

\bibitem{Chebli2004}
H.~Ch{\'{e}}bli and J.~Faraut.
\newblock {Fonctions holomorphes {\`a} croissance mod{\'e}r{\'e}e et vecteurs
  distributions}.
\newblock {\em Math. Z.}, 248(3):540--565, 2004.

\bibitem{Christensen2010}
J.~G. Christensen.
\newblock Sampling in reproducing kernel {B}anach spaces on {L}ie groups.
\newblock {\em J. Approx. Theory}, 164(1):179--203, 2012.

\bibitem{Christensen2018}
J.~G. Christensen.
\newblock Atomic decompositions of mixed norm {B}ergman spaces on tube type
  domains.
\newblock In {\em Representation theory and harmonic analysis on symmetric
  spaces}, volume 714 of {\em Contemp. Math.}, pages 77--85. Amer. Math. Soc.,
  Providence, RI, 2018.

\bibitem{Christensen2019}
J.~G. Christensen, A.~H. Darweesh, and G.~\'{O}lafsson.
\newblock Coorbits for projective representations with an application to
  {B}ergman spaces.
\newblock {\em Monatsh. Math.}, 189(3):385--420, 2019.

\bibitem{Christensen2017}
J.~G. Christensen, K.~Gr{\"{o}}chenig, and G.~{\'{O}}lafsson.
\newblock New atomic decompositions for {B}ergman spaces on the unit ball.
\newblock {\em Indiana Univ. Math. J.}, 66(1):205--235, 2017.

\bibitem{Christensen2009}
J.~G. Christensen and G.~{\'{O}}lafsson.
\newblock {Examples of coorbit spaces for dual pairs}.
\newblock {\em Acta Appl. Math.}, 107(1-3):25--48, 2009.

\bibitem{Christensen2011}
J.~G. Christensen and G.~{\'{O}}lafsson.
\newblock {Coorbit spaces for dual pairs}.
\newblock {\em Appl. Comput. Harmon. Anal.}, 31(2):303--324, 2011.

\bibitem{Christensen1996}
O.~Christensen.
\newblock {Atomic decomposition via projective group representations}.
\newblock {\em Rocky Mountain J. Math.}, 26(4):1289--1312, 1996.

\bibitem{Coifman1980}
R.~R. Coifman and R.~Rochberg.
\newblock Representation theorems for holomorphic and harmonic functions in
  ${L}^p$.
\newblock In {\em Representation theorems for {H}ardy spaces}, volume~77 of
  {\em Ast\'erisque}, pages 11--66. Soc. Math. France, Paris, 1980.

\bibitem{Cowling1978}
M.~Cowling.
\newblock The {K}unze-{S}tein phenomenon.
\newblock {\em Ann. Math. (2)}, 107(2):209--234, 1978.

\bibitem{Faraut1990}
J.~Faraut and A.~Kor{\'{a}}nyi.
\newblock {Function spaces and reproducing kernels on bounded symmetric
  domains}.
\newblock {\em J. Funct. Anal.}, 88(1):64--89, 1990.

\bibitem{Feichtinger1981}
H.~G. Feichtinger.
\newblock {A characterization of minimal homogeneous {B}anach spaces}.
\newblock {\em Proc. Amer. Math. Soc.}, 81(1):55--61, 1981.


\bibitem{Feichtinger1988}
H.~G. Feichtinger and K.~Gr{\"{o}}chenig.
\newblock {A unified approach to atomic decompositions via integrable group
  representations}.
\newblock In {\em Function spaces and applications (Lund, 1986)}, volume 1302
  of {\em Lecture Notes in Math.}, pages 52--73. Springer, Berlin, 1988.

\bibitem{Feichtinger1989}
H.~G. Feichtinger and K.~Gr{\"{o}}chenig.
\newblock {B}anach spaces related to integrable group representations and their
  atomic decompositions. {I}.
\newblock {\em J. Funct. Anal.}, 86(2):307--340, 1989.

\bibitem{Koranyi1965}
A.~Kor{\'{a}}nyi and J.~A. Wolf.
\newblock {Realization of hermitian symmetric spaces as generalized
  half-planes}.
\newblock {\em Ann. of Math. (2)}, 81:265--288, 1965.

\bibitem{Pap2012}
M.~Pap.
\newblock Properties of the voice transform of the {B}laschke group and
  connections with atomic decomposition results in the weighted {B}ergman
  spaces.
\newblock {\em J. Math. Anal. Appl.}, 389(1):340--350, 2012.

\bibitem{Rochberg1982}
R.~Rochberg.
\newblock Interpolation by functions in {B}ergman spaces.
\newblock {\em Michigan Math. J.}, 29(2):229--236, 1982.

\bibitem{Rudin1991}
W.~Rudin.
\newblock {\em Functional analysis}.
\newblock International Series in Pure and Applied Mathematics. McGraw-Hill
  Inc., New York, second edition, 1991.

\bibitem{Stoll1977}
M.~Stoll.
\newblock Mean value theorems for harmonic and holomorphic functions on bounded
  symmetric domains.
\newblock {\em J. Reine Angew. Math.}, 290:191--198, 1977.

\bibitem{Triebel1988}
H.~Triebel.
\newblock {Function spaces on {L}ie groups and on analytic manifolds}.
\newblock In {\em Function spaces and applications (Lund, 1986)}, volume 1302
  of {\em Lecture Notes in Math.}, pages 384--396. Springer, Berlin, 1988.

\bibitem{Varadarajan1985}
V.~S. Varadarajan.
\newblock {\em Geometry of quantum theory}.
\newblock Springer-Verlag, New York, second edition, 1985.

  
\end{thebibliography}
\end{document}